\documentclass[a4, 11pt]{article}
\usepackage{amsmath}
\usepackage{amsfonts}
\usepackage{mathrsfs}
\usepackage{dsfont}
\title{Large Deviation in Harnack type Dirichlet spaces}
\date{}
\author{Ann-Kathrin Jarecki}
\newtheorem{thm}{Theorem}
\newtheorem{cor}[thm]{Corollary}
\newtheorem{lemma}[thm]{Lemma}
\newtheorem{prop}[thm]{Proposition}

\newtheorem{defn}[thm]{Definition}
\newtheorem{rem}[thm]{Remark}

\numberwithin{thm}{section} \numberwithin{equation}{section}

\newenvironment{proof}{{\bf
Proof:\,}}{\hspace*{\fill}\rule{1.2ex}{1.2ex}\\ }

\newcommand{\D}{\mathds{D}}
\newcommand{\E}{\mathcal{E}}

\newcommand{\X}{\mathcal{X}}
\newcommand{\Y}{\mathcal{Y}}
\newcommand{\C}{\mathcal{C}}

\newcommand{\N}{\mathds{N}}
\newcommand{\R}{\mathds{R}}

\newcommand{\Prob}{\mathds{P}}
\newcommand{\eins}{\mathds{1}}

\DeclareMathOperator{\supp}{supp}

\DeclareMathOperator*{\esss}{esssup}
\DeclareMathOperator*{\essi}{essinf}
\DeclareMathOperator{\vol}{Vol}

\begin{document}

\maketitle

\begin{abstract}
In the framework of Harnack type Dirichlet forms, we prove a large deviation principle for the asymptotics of reversible Markov processes with rate function given by the energy of the paths.

\end{abstract}

\section{Harnack type Dirichlet spaces}

\subsection{Framework}

In the following we consider a fixed regular Dirichlet form $(\E,\D)$ with domain $\D \subset L^2(\X,m)$. The underlying topological space is a locally compact separable metric space and $m$ a positive radon measure with $\supp(m)=\X$ (such that any open relatively compact nonempty set has positive measure). Let $\{T_t\}_{t > 0}$ be the associated self--adjoint strongly continuous semigroup on $L^2(\X,m)$ and $A$ the corresponding infinitesimal generator. We assume the Dirichlet form to be strongly local, i.e. $\E(u,v)=0$ if $u,v \in \D$ have compact support and $v$ is constant on a neighbourhood of the support of $u$.

\begin{rem}
A Dirichlet form $\E$ is called regular if $\D \cap \C_c(\X)$ is dense in $\C_c(\X)$ in the sup norm $||u||_\infty = \sup_{\X}\{|u|\}$ and dense in $\D$ in the norm $\E_1(u,u)^{1/2}=\sqrt{||u||_2^2 + \E(u,u)}$. Hence there is a connection between $\E$ and the topology of $\X$.
\end{rem}

Let $\D_0=\{f \in \D_b \ :\ I_f(h)\leq \|h\|_{L^1} \mbox{ for all } h
\in \D_b=\D \cap L^{\infty}\}$. We can define in an intrinsic way a pseudo metric $d$ on $\X$

\begin{equation}\label{pseudoMet}
d(x,y)=\sup_{f \in \D_0} \big\{f(x)-f(y)\big\}.
\end{equation}

In general, $d$ may be degenerate, i.e. $d(x,y)=\infty$ or
$d(x,y)=0$ for some $x\neq y$.

In addition to the previous assumptions we further assume the Dirichlet form $\E$ to be strongly regular in the following sense:

\begin{defn} A strongly local, symmetric Dirichlet form $\E$ is called strongly regular if it is regular an if $d$ (defined by \ref{pseudoMet}) is a metric on $\X$ whose topology coincides with the original one.
\end{defn}

\begin{rem}
Strong regularity implies, that $d$ is non--degenerate, $\X$ is connected an for any $y\in \X$ the function $f \mapsto d(x,y)$ is continuous. Hence any ball $B_r(x)=\{y \in \X : d(x,y) < r\}$ is connected and its boundary coincides with the sphere $S_r(x)=\{y \in \X : d(x,y)=r\}$. For any fixed $x \in \X$ and sufficiently small $r>0$ the closed balls $\overline{B}_r(x)$ are compact and thus complete. But this not necessarily imply that all balls $B_r$ are relatively compact in $\X$. This is true if and only if the metric space $(\X,d)$ is complete.
\end{rem}

\subsection{Local weak solutions}

Identify the Hilbert space $L^2(\X,m)$ with its own dual (using the inner product). Denote the dual of $\D$ by $\D^*$. Then we have the following continuous and dense embeddings $\D \subset L^2(\X,m) \subset \D^*$.
Let $V \subset \X$ be a nonempty open subset and $I \subset \R$ a open time interval. We now want to define what is a weak local solution of the heat equation $\partial_t u = A u$ in the time--space cylinder $I \times V$. We define:

\begin{itemize}
    \item $L^2(I \to \D)$ being the Hilbert space of functions $u : I \to \D$ such that
    $$||u||_{L^2(I \to \D)}=\left(\int_I ||u_t||^2_{\D} dt\right)^{1/2} < \infty;$$

    \item $H^1(I \to \D^*) \subset L^2(I \to \D)$ being the Hilbert space of those functions $u \in L^2(I \to \D)$ whose distributional time derivative $\frac{\partial}{\partial_t} \in L^2(I \to \D^*)$ equipped with the norm
    $$||u||_{H^1(I \to \D^*)}=\left(\int_I ||u_t||^2_{\D^*} + ||\frac{\partial}{\partial_t}u_t||^2_{\D^*}dt\right)^{1/2} < \infty;$$

    \item $\D(I \times \X):= L^2(I \to \D) \cap H^1(I \to \D^*)$ being a Hilbert space with norm
    $$||u||_{\D(I \times \X)}=\left(\int_I ||u_t||^2_{\D} + ||\frac{\partial}{\partial_t}u_t||^2_{\D^*}dt\right)^{1/2} < \infty;$$

    \item $\D_{loc}(I \times V)$ being the set of all functions $u : I \times V \to \R$ such that for any open interval $I' \subset I$ relatively compact in $I$ and any open subset $V'$ relatively compact in $V$ there exists a function $u' \in \D(I \times \X)$ satisfying $u=u'$ a.e. in $I' \times V'$;

    \item $\D_c(I \times V)=\{u \in \D(I \times \X) : u_t(\cdot) \mbox{ has compact support in } V \mbox{ for a.a. } t \in I\}.$
\end{itemize}

\begin{defn}
A function $u : I \times V \to \R$ is a weal (local) solution of the heat equation $(\frac{\partial}{\partial_t} - A)u=0$ in $I \times V$ if
\begin{itemize}
    \item[(i)] $u \in \D_{loc}(I \times V$)

    \item[(ii)] for any open interval $J$ relatively compact in $I$ and $\varphi \in \D_c(I \times V)$
    $$\int_J\int_V \varphi_t \frac{\partial}{\partial_t}u_t dm~dt + \int_J\E(\varphi_t u_t) dt.$$
\end{itemize}
\end{defn}

\subsection{The Harnack inequality}

The Harnack inequality is an inequality relating the values of local solutions $u : (t,x) \to u_t(x)$ of $(A - \frac{\partial}{\partial_t})u=0$ on $\R \times \X$. Let $Y\subset \X$ an arbitrary subset.

\begin{defn}
The Harnack inequality for the operator $(A-\frac{\partial}{\partial_t})u=0$ is satisfied if there exists a constant $C=C(Y)$ such that for all $x\in \X$, all balls $B_{2r}\subset Y$ and all $t\in \R$
\begin{equation*}\tag{HI} \label{HI}
\sup_{(s,y)\in Q^-} u_s(y) \leq C \inf_{(s,y)\in Q^+} u_s(y)
\end{equation*}
whenever $u$ is a nonnegative solution of the heat equation $(A - \frac{\partial}{\partial_t})u=0$ on $Q=]t-4r^2,t[\times B_{2r}(x)$. Here $Q^-=]t-3r^2,t-2r^2[\times B_{r}(x)$ and $Q^+=]t-r^2,t[\times B_{r}(x)$.

\end{defn}

\begin{rem}
Actually one has to replace the $\inf$ and $\sup$ in (\ref{HI}) by $\essi$ and $\esss$, i.e. the estimate is correct up to sets of measure zero. But one consequence of the Harnack inequality is a quantitative H\"older inequality, thus all weak solutions admit a continuous version, so it is not necessary to use $\essi$ or $\esss$.
\end{rem}

\begin{defn}
Let $(\X,m)$ as before, $Y\subset \X$ an arbitrary subset and $(\E,\D)$ a strongly regular and strictly local Dirichlet form. We assume that for all balls $B_{2r}(x)\subset Y$ the closed balls $\overline{B}_r(x)$ are complete. If additionally $(\E,\D)$ satisfies the Harnack inequality (\ref{HI}) we call $(\X,m,\E,\D)$ a \textsc{Harnack type Dirichlet space}.
\end{defn}

Before we are able to state an important consequence of the Harnack inequality we have to establish the following

\begin{defn}
Let $(\X,m)$ as above an $(\E,\D)$ a strongly regular and strictly local Dirichlet form and $Y \subset \X$ an arbitrary fixed subset. Then we say
    \begin{itemize}
    \item the \textsc{doubling volume property} holds if there is a constant $N=N(Y)$ such that for all balls $B_{2r}(x) \subset Y$
    \begin{equation*}\tag{VD} \label{VD}
    m(B_{2r})\leq 2^N m(B_r(x)).
    \end{equation*}

    \item the \textsc{weak Poincar\'{e} inequality} holds if there exists a constant $\kappa=\kappa(Y)$ such that for all balls $B_{2r} \subset Y$
    \begin{equation*}\tag{PI}\label{PI}
    \int_{B_r(x)}|u-u_{r,x}|^2 dm \leq \kappa ~ r^2 \int_{B_{2r}(x)}d\Gamma(u,u)
    \end{equation*}
    for all $u \in \D$ where $u_{r,x}=\frac{1}{m(B_r(x))} \int_{B_r(x)u dm}$ denotes the average of $u$ over $B_r(x)$.
    \end{itemize}
\end{defn}

Now we are able to state the following theorem (cf. \cite{SIII})

\begin{thm}\label{equi}
Assume that for all balls $B_{2r}(x)\subset Y$ the closed balls $\overline{B}_{r}(x)$ are compact then the following are equivalent:
    \begin{itemize}
    \item[(i)] The volume doubling property (\ref{VD}) and the Poincar\'{e} inequality (\ref{PI}) hold true on $Y$.

    \item[(ii)] The parabolic Harnack inequality (\ref{HI}) holds true for the operator $(L-\frac{\partial}{\partial_t})$ on $\R \times Y$ holds true.
    \end{itemize}
\end{thm}

\begin{lemma}\label{LDPVD}
For $m$--a.e. $x\in \X$ and all $\varepsilon \in (0,1]$, such that $\vol(\sqrt{\varepsilon},x):=m\big[B_{\sqrt{\varepsilon}}(x)\big]>0$ we have
$$\lim_{t\searrow 0} t \log \vol(\sqrt{\varepsilon t},x) =0.$$
\end{lemma}

\begin{proof}
Since $(\X,m,\E,\D)$ is a Harnack type Dirichlet form we know from theorem (\ref{equi}) that the volume doubling property (\ref{VD}) holds, i.e. it exists a constant $N=N(Y)$ such that for all balls $B_{2r}(x)\subset Y$: $\vol(2r,x)\leq 2^N \vol(r,x)$. This implies
$$\vol(r',x)\leq \left(\frac{r'}{r}\right)^N \vol(r,x)$$
for all $B_{r}(x)\subset B_{r'}(x)\subset Y$. Hence we get for all $0<t\leq 1$
$$\vol(\sqrt{\varepsilon t},x) \geq t^{N/2} \vol(\sqrt{\varepsilon},x).$$
Because $\vol(\sqrt{\varepsilon t},x) \in [0,1]$ we see
$$0\geq \lim_{t \searrow 0} t \log \vol(\sqrt{\varepsilon t},x) \geq \lim_{t \searrow 0} \frac{N}{2} t \log t \stackrel{l'Hopital}{=} 0.$$
\end{proof}

\begin{rem}
Analogous
$$\lim_{t \searrow 0} t \log \left(\vol(\sqrt{\varepsilon t},x)\right)^{-1}=0$$
since $\log \left(\vol(\sqrt{\varepsilon t},x)\right)^{-1}= - \log \vol(\sqrt{\varepsilon t},x)\geq 0$.
\end{rem}

\section{Upper and lower bound for the heat kernel}

Let $(\X,m,\E,\D)$ be a Harnack type Dirichlet space. According to theorem (\ref{equi}) this implies the volume doubling property (\ref{VD}) and the Poincar\'{e} inequality (\ref{PI}). Under this assumptions it is possible to derive pointwise estimates for the density of the semigroup $\{T_t\}_{t>0}$ as well as for the fundamental solution $p_t(x,y)$ of the operator $A-\frac{\partial}{\partial_t}$.

For the following proposition see \cite{SII}

\begin{prop}\label{fundsolu}
There exists a measurable function $p : \R\times \X \times \X \to [0,\infty[$ with the following properties:

\begin{itemize}
    \item[(i)] for every $t>0$, $m$--a.e. $x,y \in \X$ and every $f \in L^1(\X,m)+L^\infty(\X,m)$
    $$T_t f(y)=\int_\X p_t(z,y)f(z)m(dz);$$

    \item[(ii)] for every $0< \sigma < \tau$ and $m$--a.e. $x\in \X$ the function
    $$u: (t,y)\mapsto p_t(x,y)$$
    is a global solution of the equation $Au=\frac{\partial}{\partial_t}u$ on $]\sigma,\tau[\times \X$;

    \item[(iii)] for every $0<r<t$ and $m$--a.e. $x,y \in \X$
    $$p_t(x,y)=\int_\X p_r(x,z)p_{t-r}(z,y)m(dz);$$

    \item[(iv)] for every $0<t$
    $$\int_\X\int_\X p_t(x,y)^2m(dx)m(dy)=||T_t||^2\leq 1.$$
\end{itemize}

\end{prop}

\begin{rem}
If $(\X,m,\E,\D)$ is a Harnack type Dirichlet space then $\{T_t\}_{t>0}$ -- as a solution of $(A-\frac{\partial}{\partial_t})u=0$ -- satisfies the Harnack inequality (\ref{HI}). From proposition (\ref{fundsolu}) (ii) it follows that also $p_t(x,y)$ satisfies the Harnack inequality on $]\sigma,\tau[\times \X$ for $m$--a.e. $x\in \X$ and all $0<\sigma<\tau$, i.e. for all $t>0$ and $B_{2r}(z)\subset Y$ it exists a constant $C=C(Y)$ such that
$$\sup_{(s,y)\in Q^-}p_s(x,y) \leq C \inf_{(s,y)\in Q^+}p_s(x,y)$$
if $]t-4r^2,t[\subset ]\sigma,\tau[$. Here $Q^-=]t-3r^2,t-2r^2[\times B_r(z)$ and $Q^+=]t-r^2,t[\times B_r(z)$.
\end{rem}

For another useful theorem see \cite{SIII}.

\begin{thm}\label{thmIII}
Let $(\X,m,\E,\D)$ a Harnack type Dirichlet space (for $Y \subset\X$). Then there exists a constant such that

$$p_t(x,y) \geq \frac{1}{C} \cdot \vol(\sqrt{s},x)^{-1}\cdot \exp\left(-C\frac{d^2(x,y)}{t}\right)\cdot \exp\left(-t\frac{C}{R^2}\right)$$

for $t>0$ and all points $x,y \in Y$ which are joined in Y by a curve $\gamma$ of length $d(x,y)$. Here $s=\inf\{t,R^2\}$ with $R=\inf_{0\leq l \leq 1}d(\gamma(l),\X\backslash Y)$ (being $+\infty$ if $\X=Y$).

\end{thm}

\subsection{Upper bound}

In this section we want to prove (pointwise) upper bounds for $T_t \eins_A(x)$ and $p_t(x,y)$. For it we use the integrated Gaussian estimates
\begin{equation}\label{gaussian}
(\eins_A,T_t\eins_B)_{L^2} \leq \sqrt{m(A)m(B)} \exp\left[-\frac{d^2(A,B)}{2t}\right],
\end{equation}
with $t>0$ and $A,B \subset \X$ measurable subsets. For references see for example \cite{SII} and \cite{HR}. Together with the Harnack inequality and proposition (\ref{LDPVD}) we can deduce the following lemma:

\begin{lemma}
The following pointwise estimates hold

\begin{itemize}
    \item[(i)] for all measurable sets $A\subset \X$ and $m$--a.e. $x \in \X$
    $$\limsup_{t \searrow 0}t \log T_t\eins_A(x) \leq - \frac{d^2(A,x)}{2};$$

    \item[(ii)] for $m$--a.e. $x,y \in \X$
    $$\limsup_{t \searrow 0}t \log p_t(x,y) \leq - \frac{d^2(x,y)}{2}.$$
\end{itemize}
\end{lemma}

\begin{proof}

\begin{eqnarray*}
T_t \eins_A(x) &\leq & \sup_{y\in B_{\sqrt{\varepsilon t}}(x)} T_t\eins_A(y) \ \stackrel{HI}{\leq}\ C \inf_{y\in B_{\sqrt{\varepsilon t}}(x)}T_{t(1+2\varepsilon)}\eins_A(y)\\
&\leq& C \cdot \vol(\sqrt{\varepsilon t},x)^{-1}\int_{B_{\sqrt{\varepsilon t}}(x)} T_{t(1+2\varepsilon)} \eins_A (z)m(dz)\\
&=& C \cdot \vol(\sqrt{\varepsilon t},x)^{-1} \left(\eins_{B_{\sqrt{\varepsilon t}}(x)}, T_{t(1+2\varepsilon)}\eins_A\right)_{L^2} \\
&\leq&C \cdot t^{-N/2} \cdot \vol(\sqrt{\varepsilon},x)^{-1} \cdot \sqrt{m[A] m[B_{\sqrt{\varepsilon t}}(x)]}\exp\left[-\frac{d^2(A,B_{\sqrt{\varepsilon t}}(x))}{2t(1+2\varepsilon)}\right]\\
&\leq&C \cdot t^{-N/2} \cdot \vol(\sqrt{\varepsilon},x)^{-1/2} \cdot \sqrt{m[A]}\exp\left[-\frac{d^2(A,B_{\sqrt{\varepsilon t}}(x))}{2t(1+2\varepsilon)}\right]
\end{eqnarray*}

So for all $\varepsilon >0$ we get ($d$ continuous) with lemma \ref{LDPVD}

$$\limsup_{t \searrow 0}t \log T_t \eins_A(x)\leq-\frac{d^2(A,x)}{2(1+2\varepsilon)},$$

and (i) follows for $\varepsilon \searrow 0$.

For part (ii) we apply the Harnack inequality twice

\begin{eqnarray*}
p_t(x,y)&\leq& C \cdot \vol(\sqrt{\varepsilon t},y)^{-1} \underbrace{\int_{B_{\sqrt{\varepsilon t}}(y)} p_{t(1+2\varepsilon)}(x,z)m(dz)}_{=T_{t(1+2\varepsilon)}\eins_{B_{\varepsilon t}}(x)}\\
&\leq &C^2 \cdot \vol(\sqrt{\varepsilon t},x)^{-1} \cdot \vol(\sqrt{\varepsilon t},y)^{-1}\cdot \int_{B_{\sqrt{\varepsilon t}}(x)} T_{t(1+4\varepsilon)} \eins_{B_{\sqrt{\varepsilon t}}(y)} (z)m(dz)\\
&\leq& C^2 \cdot t^{-N} \cdot \vol(B_{\sqrt{\varepsilon}}(x))^{-1/2} \cdot \vol(B_{\sqrt{\varepsilon}}(x))^{-1/2} \cdot \exp\left[-\frac{d^2(B_{\sqrt{\varepsilon t}}(x),B_{\sqrt{\varepsilon t}}(y))}{2t(1+4\varepsilon)}\right]
\end{eqnarray*}

As before for $\varepsilon >0$ we get

$$\limsup_{t \searrow 0}t \log p_t(x,y)\leq-\frac{d^2(x,y)}{2(1+4\varepsilon)},$$

and hence (ii).

\end{proof}

\subsection{Lower bound}

The goal of this section is to prove the corresponding lower bound. Therefor we use theorem (1.1) in \cite{HR}, which states that $\liminf_{t \searrow 0} t \log \left(\eins_A,T_t\eins_B\right)_{L^2} \geq -\frac{d^2(A,B)}{2}$. Additionally we use theorem (\ref{thmIII}) to show the following lemma.
\begin{lemma}
The following pointwise estimates hold:
\begin{itemize}
    \item[(i)] for all measurable sets $A\subset \X$ and $m$--a.e. $x\in \X$
    \begin{equation}\label{LDPTt}
    \liminf_{t\searrow 0} t \log T_t \eins _A(x) \geq - \frac{d^2(A,x)}{2};
    \end{equation}

    \item[(ii)] for $m$--a.e. $x,y \in \X$
    \begin{equation}\label{pt}
    \liminf_{t\searrow 0} t \log p_t(x,y) \geq - \frac{d^2(x,y)}{2}.
    \end{equation}
\end{itemize}
\end{lemma}

\begin{proof}
(i)

\begin{eqnarray}\label{Ttunten1}
T_t\eins_A (x) &= & \int_A p_t(z,x) m(dz)\nonumber\\
&\geq& \int_A \int_{B_{\varepsilon}(x)} p_{t(1-\varepsilon)}(z,z')p_{t\varepsilon}(z',x)m(dz')m(dz).
\end{eqnarray}

Now we can use theorem (\ref{thmIII}) to estimate $p_{t\varepsilon}(z',x)$. If we choose $\varepsilon$ small enough we know that $\overline{B}_{\varepsilon}$ is compact and thus complete. Then lemma (1.2) in \cite{SIII} tells us that every $z'\in \overline{B}_{\varepsilon}$ can be joined with the center $x$ by a minimal geodesic in $\overline{B}_{\varepsilon}$. Hence the requirements of theorem (\ref{thmIII}) are fulfilled and we get for all $z' \in B_{\varepsilon}$

\begin{eqnarray}\label{pepsilont}
p_{\varepsilon t}(z',x)&\geq& \frac{1}{C}\cdot \vol(\sqrt{\varepsilon t},x)^{-1}\cdot \exp\left(- C\frac{d^2(x,z')}{\varepsilon t}\right)\nonumber\\
&\geq& \frac{1}{C}\cdot \vol(\sqrt{\varepsilon t},x)^{-1}\cdot \exp\left(- C\frac{\varepsilon}{t}\right).
\end{eqnarray}

Thus together with \ref{Ttunten1} we get

\begin{eqnarray}
\lefteqn{\liminf_{t\searrow 0} t \log T_t\eins_A (x) \geq \int_A \int_{B_{\varepsilon}(x)} p_{t(1-\varepsilon)}(z,z')p_{t\varepsilon}(z',x)m(dz')m(dz)}\nonumber \\
& \geq & \liminf_{t\searrow 0} t \log \left[\frac{1}{C}\cdot \vol(\sqrt{\varepsilon t},x)^{-1}\cdot \exp\left(- C\frac{\varepsilon}{t}\right) \int_A \int_{B_{\varepsilon}(x)} p_{t(1-\varepsilon)}(z,z')m(dz')m(dz)\right]\nonumber\\
& \geq & \liminf_{t\searrow 0} t \log \left[\frac{1}{C}\cdot \vol(\sqrt{\varepsilon},x)^{-1}\cdot \exp\left(- C\frac{\varepsilon}{t}\right) \left(\eins_A,T_{t(1-\varepsilon)}\eins_{B_{\varepsilon}(x)}\right)_{L^2}\right]\nonumber\\
&\geq& -\frac{d^2(A,B_{\varepsilon}(x))}{2(1-\varepsilon)} - C\varepsilon.
\end{eqnarray}

Since $d$ is continuous $-\frac{d^2(A,B_{\varepsilon}(x))}{2(1-\varepsilon)}$ converges to $-\frac{d^2(A,x)}{2}$ as $\varepsilon$ tends to zero we obtain

\begin{equation}
\liminf_{t\searrow 0} t \log T_t \eins_A(x)\geq -\frac{d^2(A,x)}{2}.
\end{equation}

(ii) From proposition (\ref{fundsolu}) we know

\begin{eqnarray}\label{propiii}
p_t(x,y) &=& p_{t(1-\varepsilon)+t\varepsilon}(x,y)\nonumber\\
&=&\int_\X p_{t(1-\varepsilon)}(z,y)p_{t\varepsilon}(x,z)m(dz)\nonumber\\
&\geq& \int_{B_{\varepsilon}(x)} p_{t(1-\varepsilon)}(z,y)p_{t\varepsilon}(x,z)m(dz).
\end{eqnarray}

Like in the proof of part (i) (cf. (\ref{pepsilont})) we have the following estimate
\begin{eqnarray*}
p_{\varepsilon t}(x,z)&\geq& C \cdot (\vol(\sqrt{\varepsilon t},x))^{-1} \cdot \exp\left[-C \frac{d^2(x,y)}{\varepsilon t}\right]\\
&\geq& C \cdot (\vol(\sqrt{\varepsilon t},x))^{-1} \cdot \exp\left[-C \frac{\varepsilon}{t}\right].
\end{eqnarray*}

Together with (\ref{propiii}) we obtain

\begin{eqnarray*}
p_{t}(x,y)&\geq& \int_{B_{\varepsilon}(x)} p_{t(1-\varepsilon)}(z,y)p_{t\varepsilon}(x,z)m(dz)\\
&\geq& C \cdot (\vol(\sqrt{\varepsilon t},x))^{-1} \cdot T_{t(1-\varepsilon)} \eins_{B_\varepsilon(x)}(y) \cdot \exp\left[-C \frac{\varepsilon}{t}\right].
\end{eqnarray*}

Applying part (i) of this lemma and lemma (\ref{LDPVD}) we get

\begin{eqnarray*}
\liminf_{t \searrow 0} t \log p_t(x,y)&\geq& \liminf_{t \searrow 0} t \log \left[T_{t(1-\varepsilon)} \eins_{B_\varepsilon(x)}(y)\right]+ \liminf_{t \searrow 0} t \log \left[\exp\left[-C \frac{\varepsilon}{t}\right]\right]\\
&=& -\frac{d^2(B_\varepsilon (x),y)}{2(1-\varepsilon)} - C \cdot \varepsilon.
\end{eqnarray*}

This holds for all $\varepsilon >0$ and thus we get for $\varepsilon \searrow 0$

$$\liminf_{t \searrow 0} t \log p_t(x,y) \geq -\frac{d^2(x,y)}{2}.$$

\end{proof}

\section{Upper and lower bound for finite dimensional distribution of the associated Markov process}

In this section we analyse the short time behaviour of the finite dimensional distributions of the Markov process associated to a Harnack type Dirichlet space $(\X,m,\E,\D)$. For this we use the pointwise estimates of the last section. The goal is to proof the following theorem

\begin{thm}\label{endldim}
Let $(\X,m,\E,\D)$ be a Harnack type Dirichlet space. Let $X_t$ be the Markov process associated to our Dirichlet form $\E$ on the probability space $(\Omega,\Prob),\ X : \Omega \to \C([0,1],X)$. For all partitions $\Delta^n=\{0=t_0<t_1<\ldots<t_n=1\}$ of the unit interval $[0,1]$ and for all $\mathcal{A}=(A_0,A_1,\ldots,A_n)\in \X^{n+1}$ we get
\begin{itemize}

\item[(i)]
    \begin{equation*}
    \liminf_{s \searrow 0} s \log \Prob\left(X_{s\cdot t_0} \in A_0,X_{s\cdot t_1} \in A_1,\ldots,X_{s\cdot t_n}\in A_n\right)\geq \sup_{\overline{x} \in \mathcal{A}^{0}}{}-{}\sum_{i=0}^{n-1}\frac{d^2(\overline{x}_{t_i},\overline{x}_{i+1})}{2(t_{i+1}-t_i)} \end{equation*}

\item[(ii)]
    \begin{equation*}
    \limsup_{s \searrow 0} s \log \Prob\left(X_{s\cdot t_0} \in A_0,X_{s\cdot t_1} \in A_1,\ldots,X_{s\cdot t_n}\in A_n\right)\leq \sup_{\overline{x} \in \overline{\mathcal{A}}}{}-{}\sum_{i=0}^{n-1}\frac{d^2(\overline{x}_{t_i},\overline{x}_{i+1})}{2(t_{i+1}-t_i)} \end{equation*}

\end{itemize}
\end{thm}

\begin{proof}
(i) Lower bound:\\[0.2ex]
First we define for a subset $A \subset \X$ and $\beta > 0$ the open set

$$A^{\beta_-}:=\{x \in \X : d(A^C,x)> \beta\}.$$

Further let $\delta_i:=t_{i}-t_{i-1}.$

Fix $\beta>0$ then for all $\varepsilon >0,\ \beta \geq \sqrt{\varepsilon}$ and $m$--a.e. $\overline{x}_i \in A^{\beta_-}_i$

\begin{eqnarray*}
\lefteqn{\Prob\left(X_{t_0} \in A_0,X_{t_1} \in A_1,\ldots,X_{t_n} \in A_n\right)}\\
&\geq& \int_{A_0}\int_{B_{\sqrt{\varepsilon t}}(\overline{x}_1)}\int_{A_2}\ldots\int_{A_n}p_{\delta_1}(x_0,x_1)\cdot p_{\delta_2}(x_1,x_2)\cdot \ldots\\
    &&\hspace{7cm} \ldots \cdot p_{\delta_n}(x_{n-1},x_n)m(dx_n)\ldots m(dx_1)m(dx_0)\\
&\geq& \frac{1}{C^2}\vol(\sqrt{\varepsilon t},\overline{x}_1)\int_{A_0}\int_{A_2}\ldots\int_{A_n}p_{\delta_1(1-2\varepsilon)}(x_0,\overline{x}_1)\cdot p_{\delta_2(1-2\varepsilon)}(\overline{x}_1,x_2)\cdot p_{\delta_3}(x_{2},x_3)\cdot \ldots\\
    &&\hspace{7cm} \ldots \cdot p_{\delta_n}(x_{n-1},x_n)m(dx_n)\ldots m(dx_2)m(dx_0)\\
&\vdots&\\
&\geq&\frac{1}{C^{2(n-1)}}\vol(\sqrt{\varepsilon t},\overline{x}_1)\cdot \ldots \cdot \vol(\sqrt{\varepsilon t},\overline{x}_{n-1})\cdot p_{\delta_2(1-4\varepsilon)}(\overline{x}_1,\overline{x}_2)\cdot \ldots \cdot p_{\delta_{n-1}(1-4\varepsilon)}(\overline{x}_{n-2},\overline{x}_{n-1})\\
    &&\hspace{4.5cm}\cdot \underbrace{\int_{A_0} p_{\delta_1(1-2\varepsilon)}(x_0,\overline{x}_1) m(dx_0)}_{=T_{\delta_1(1-2 \varepsilon)}\eins_{A_0}(\overline{x}_1)}\cdot \underbrace{\int_{A_n}p_{\delta_n(1-2\varepsilon)}(\overline{x}_{n-1}x_n)m(dx_n)}_{=T_{\delta_n(1-2 \varepsilon)}\eins_{A_n}(\overline{x}_{n-1})}
\end{eqnarray*}

Since this holds true for $m$--a.e. $\overline{x}=(\overline{x}_1,\overline{x}_1,\overline{x}_2,\ldots,\overline{x}_n) \in \mathcal{A^{\beta_-}}:=A^{\beta_-}_0 \times A^{\beta_-}_1 \times \ldots \times A^{\beta_-}_n$ we get

\begin{eqnarray*}
\lefteqn{\liminf_{s\searrow 0} s \log \Prob\left(X_{s\cdot t_0} \in A_0,X_{s\cdot t_1} \in A_1,\ldots,X_{s\cdot {t_n}} \in A_n\right)}\\
&\geq& \max_{\overline{x}\in \mathcal{A}^{\beta_-}} {}-{}\frac{d^2(A_0,\overline{x}_1)}{2\delta_1(1-2\varepsilon)} - \sum_{i=1}^{n-2} \frac{d^2(\overline{x}_i,\overline{x}_{i+1})}{2\delta_{i+1}(1-4\varepsilon)} -\frac{d^2(\overline{x}_{n-1},A_n)}{2\delta_n(1-2\varepsilon)}\\
&\geq&\max_{\overline{x}\in \mathcal{A}^{\beta_-}} {}-{} \sum_{i=0}^{n-1} \frac{d^2(\overline{x}_i,\overline{x}_{i+1})}{2\delta_{i+1}(1-4\varepsilon)} .
\end{eqnarray*}

If $\varepsilon \searrow 0,\ \beta \searrow 0$ we obtain

\begin{equation*}
\liminf_{s\searrow 0} s \log \Prob\left(X_{s\cdot t_0} \in A_0,X_{s\cdot t_1} \in A_1,\ldots,X_{s\cdot {t_n}} \in A_n\right) {}\geq {}\max_{\overline{x}\in \mathcal{A}^0} {}-{} \sum_{i=0}^{n-1} \frac{d^2(\overline{x}_i,\overline{x}_{i+1})}{2(t_{i+1}-t_i)}.
\end{equation*}

(ii) Upper bound:\\[0.2ex]
The proof of the upper bound works nearly the same way. The only difference is that we have to consider for a subset $A \subset \X$ the following open sets
$$A^{\beta_+}:=\{x \in \X : d(A,x)< \beta\}.$$
As in part (i) we get

\begin{equation*}
\limsup_{s\searrow 0} s \log \Prob\left(X_{s\cdot t_0} \in A_0,X_{s\cdot t_1} \in A_1,\ldots,X_{s\cdot {t_n}} \in A_n\right)
{} \leq {} \max_{\overline{x} \in \mathcal{A}^{\beta_+}} {}-{} \sum_{i=0}^{n-1} \frac{d^2(\overline{x}_i,\overline{x}_{i+1})}{2\delta_{i+1}(1-4\varepsilon)} .
\end{equation*}

As before let $\varepsilon \searrow 0,\ \beta \searrow 0$ to obtain

\begin{equation*}
\limsup_{s\searrow 0} s \log \Prob\left(X_{s\cdot t_0} \in A_0,X_{s\cdot t_1} \in A_1,\ldots,X_{s\cdot {t_n}} \in A_n\right)
{} \leq {} \max_{\overline{x} \in \overline{\mathcal{A}}} {}-{} \sum_{i=0}^{n-1} \frac{d^2(\overline{x}_i,\overline{x}_{i+1})}{2(t_{i+1}-t_i)} .
\end{equation*}

\end{proof}

\section{Short-time behaviour of the Markov process}

In the last section we estimate the finite dimensional distributions of the Markov process. Now we want to lift up this result to an estimate of the short-time behaviour of the law of $X_t$ itself. Therefor we will use the theorem of Dawson--G\"artner. This theorem yields the large deviation principle in a space $\mathcal{Y}$ as a consequence of the LDP's in $\mathcal{Y}_i$, where $\mathcal{Y}$ is the projective limit of the projective system $\mathcal{Y}_i$.

To formulate the theorem of Dawson--G\"artner precisely we have to
recall some well known concepts. We mention that a LDP describes the
asymptotic behaviour, as $\varepsilon \to \infty$, of a family of
probability measures $\{\mu_\varepsilon\}$ on $(\Omega,\mathcal{B})$
in terms of a rate function, where a rate function is defined as
follows.

\begin{defn}
A function $I : \Omega \to [0,\infty]$ is called a rate function if
it is lower semi--continuous.

We say that a function $I : \Omega \to [0,\infty]$ is a good rate
function, if $I$ is lower semi--continuous and for all $\alpha \in
[0,\infty)$ the level sets $\psi_I(\alpha)=\{x \in \Omega : I(x)\leq \alpha\}$ are
compact subsets of $\Omega$.
\end{defn}

For any set $\Gamma$, $\overline{\Gamma}$ denotes the closure of
$\Gamma$ and $\Gamma^\circ$ the interior of $\Gamma$. Then we say

\begin{defn}
The family $\{\mu_\varepsilon\}$ of probability measures satisfies
the LDP with good rate function $I$ if, for all subsets $\Gamma \in
\mathcal{B}$,

$$-\inf_{\omega \in \Gamma^\circ}I(\omega) \leq \liminf_{\varepsilon \to 0}
\varepsilon \log \mu_\varepsilon (\Gamma)
\leq \limsup_{\varepsilon \to 0} \varepsilon \log \mu_\varepsilon
(\Gamma) \leq -\inf_{\omega \in \overline{\Gamma}}I(\omega).$$

The infimum of a function over an empty set is interpreted as
$\infty$.
\end{defn}

There is an other weaker form of a LDP where the upper bound is proven only for compact sets.

\begin{defn}
A family of probability measures $\{\mu_\varepsilon\}$ is said to satisfy the weak LDP with rate function $I$ if the upper bound
\begin{equation}
\limsup_{\varepsilon \to 0} \varepsilon \log \mu_{\varepsilon}(\Gamma)\leq -\alpha
\end{equation}
holds for all $\alpha < \infty$ and all compact subsets $\Gamma$ of the complement of level sets $\psi_I(\alpha)^C$ and the lower bound
\begin{equation}
\liminf_{\varepsilon \to 0} \varepsilon \log \mu_{\varepsilon}(\Gamma)\geq -I(x)
\end{equation}
holds for any $x\in \{y : I(y)<\infty\}$ and all measurable $\Gamma$ with $x \in \Gamma^{\circ}$.
\end{defn}

Let $J$ be a partial ordered set and $\{(\Y_j,p_{ij})\}_{i\leq j \in
\N}$ be a projective system, i.e. $\{\Y_j\}_{j \in J}$ is a family
of Hausdorff topological spaces and the continuous maps $p_{ij} :
\Y_j \to \Y_i$ satisfy $p_{ik}=p_{ij}\circ p_{jk}$ for all $i\leq
j\leq k$. Let $\Y = \lim\limits_{\longleftarrow}\Y_j$ be the
projective limit of this system, that is $\Y$ consists of all the
elements $\mathbf{y}=(y_j)_{j\in J}$ for which $y_i=p_{ij}(y_j)$
whenever $i<j$. Further let $p_j : \Y \to \Y_j$ the canonical continuous projections of $\Y$ on the values at $t_0,t_1,\ldots,t_n$ for the partitions $j=\{0=t_0<t_1\ldots<t_n\}$. Then the statement of the theorem of
Dawson--G\"artner reads as

\begin{thm} (Dawson--G\"artner) (cf. \cite{dembo}) \label{dawsongaertner}\\
Let $\{\mu_\varepsilon\}$ be a family of probability measures on
$\Y$. Assume that, for each $j\in J$, the family of push--forward
measures $\{{p_j}_* \mu_\varepsilon\}$ on $\Y_j$ satisfy the LDP
with good rate function $I_j : \Y_j \to [0,\infty]$. Then the family
$\{\mu_\varepsilon\}$ satisfies the LDP on $\Y$ with good rate
function $I : \Y \to [0,\infty]$ given by

$$I(\mathbf{y})=\sup_{j \in J}\{I_j(p_j(\mathbf{y}))\},\quad \mathbf{y} \in \Y.$$

\end{thm}

\begin{rem}
For the lower bound it is not necessary to assume the functional $I$ to be a good rate function, i.e. we do not have to assume that all the level sets are compact. On the other hand for the upper bound it is crucial assumption that they are all compact.
\end{rem}

To abolish having not a good rate function we can formulate the following corollary

\begin{cor} \label{weakdawsongaertner}
Let $\{\mu_\varepsilon\}$ be a family of probability measures on
$\Y$. Assume that, for each $j\in J$, the family of push--forward
measures $\{{p_j}_* \mu_\varepsilon\}$ on $\Y_j$ satisfy the weak LDP
with rate function $I_j : \Y_j \to [0,\infty]$. Then the family
$\{\mu_\varepsilon\}$ satisfies the weak LDP on $\Y$ with rate
function $I : \Y \to [0,\infty]$ given by

$$I(\mathbf{y})=\sup_{j \in J}\{I_j(p_j(\mathbf{y}))\},\quad \mathbf{y} \in \Y.$$

\end{cor}

\begin{proof}
The proof works most like the proof of the theorem (\ref{dawsongaertner}) of Dawson and G\"artner, for the lower bound it is exactly the same. For the upper bound first we get $\psi_{I_i}(\alpha)=p_{ij}\left(\psi_{I_i}(\alpha)\right)$ for all $i<j$ because all of the level sets $\psi_{I_j}(\alpha)$ of $I_j$ are closed subsets of $\mathcal{Y}_j$. Hence we get
$$\psi_I(\alpha)= \lim_{\longleftarrow} \psi_{I_j}(\alpha),$$
and $\psi_I(\alpha)$ as the projective limit of closed sets is itself a closed subset of $\mathcal{Y}$.

Now we take a compact subset $\Gamma \subset \mathcal{Y}$ and consider the projections $\Gamma_j:=p_j(\Gamma)$, since $p_j: \mathcal{Y}\to \mathcal{Y}_j$ is continuous this sets are also compact and we get
$$\Gamma=\lim_{\longleftarrow} \Gamma_j$$
and consequently
$$\Gamma \cap \psi_I(\alpha)=\lim_{\longleftarrow} \left( \Gamma_j \cap \psi_{I_j}(\alpha)\right).$$
For all $\alpha > 0$ and all compact subsets $\Gamma$ of $\psi_{I}(\alpha)^C$ (i.e. $\Gamma \cap \psi_I(\alpha)=\emptyset$) we have $\Gamma_j \cap \psi_{I_j}(\alpha)=\emptyset$ for some $j\in J$ (cf. theorem B.4 in (\cite{dembo})). Thus we get
$$\limsup_{\varepsilon \to 0} \varepsilon \log \mu_{\varepsilon}(\Gamma) \leq \limsup_{\varepsilon \to 0} \varepsilon \log \mu_{\varepsilon} \circ p_j^{-1}(\Gamma_j)\leq -\alpha.$$
\end{proof}

Coming back to the previous situation we define a discrete version of the energy functional of a curve. This energy functional will play the role of the rate function in the last corollary \ref{weakdawsongaertner}.

\begin{defn}
Let $\Delta^n=\{0=t_1 < t_1< \ldots t_n=1\}$ be a partition of the unit interval $[0,1]$, then the discrete energy functional $H_{\Delta^n} : \X^{n+1} \to [0,\infty)$ is defined by
\begin{equation}\label{discenergy}
H_{\Delta^n}(\overline{x}):=\frac{1}{2} \sum_{i=0}^{n-1} \frac{d^2(x_i,x_{i+1})}{t_{i+1}-t_i}
\end{equation}
for all $\overline{x} \in \X^{n+1}$.
\end{defn}

Now we can also define the energy of a curve $\gamma \in \Omega=\mathcal{C}([0,1],\X)$

\begin{defn}
For all $\gamma \in \Omega$ we define the energy $H : \Omega \to [0,\infty)$ of $\gamma$ by
\begin{equation}\label{enery}
H(\gamma):=\sup_{\Delta^n} H_{\Delta^n}(p_{\Delta^n}(\gamma)),
\end{equation}
where the supremum is taken over all partitions $\Delta^n$ of the unit interval $[0,1]$ and $p_{\Delta^n}(\gamma)=(\gamma(t_0),\gamma(t_1),\ldots ,\gamma(t_n))\in \X^{n+1}$.
\end{defn}

We are now able to describe the short time behaviour of the law of the Markov process $X_t$. Theorem \ref{endldim} states that the finite dimensional distributions $\Prob\left(X_{s\cdot t_0} \in \cdot,X_{s\cdot t_1} \in \cdot,\ldots,X_{s\cdot t_n}\in \cdot\right)$ satisfy the weak LDP with the discrete energy functional as rate function. Then corollary \ref{weakdawsongaertner} gives us the weak LDP for the law of the Markov process itself with rate function $H$. To apply corollary \ref{weakdawsongaertner} consider
$$J=\bigcup_{n=0}^\infty\{\Delta^n : \Delta^n=\{0=t_0<t_1<\ldots<t_n=1\}\mbox{ partition of }[0,1]\}.$$
A partial ordering on $J$ is induced by inclusion.

So at the end of this section we obtain the following theorem which is a essential part of our main theorem
\begin{thm}
Let $\Delta^n$ as above an arbitrary partition of the unit interval $[0,1]$ and for all $s >0$ let $X^s_\cdot=\{X_{s\cdot t}\}_{0\leq t \leq 1}$ be the rescaled Markov process. Then we have the following estimates
\begin{itemize}
\item[(i)] For all $\alpha > 0$ and all compact subsets $\Gamma$ of $\{\gamma \in \Omega : H(\gamma) > \alpha\}$ we have
$$\limsup_{s \to 0} s \log \Prob(X^s_\cdot \in \Gamma) \leq - \alpha.$$
\item[(ii)] For all $\gamma \in \{\gamma : H(\gamma)<\infty\}$ and all measurable $\Gamma$ with $\gamma \in \Gamma^{\circ}$ we have
$$\limsup_{s \to 0} s \log \Prob(X^s_\cdot \in \Gamma) \geq - H(\gamma).$$
\end{itemize}
\end{thm}

\subsection{Identification of the Energy Functional}
%
%
%
%
In the previous part of this section we have seen, that the weak LDP holds for the law of the Markov process with rate function $H$. In the following we want to get a more explicit expression for the energy $H$. For this we consider absolutely continuous curves $\gamma \in AC^2([0,1],\X)$ with finite $2$-energy. This are curves for which exists $m \in L^2([0,1])$ such that
\begin{equation}\label{ac2}
d(\gamma(s),\gamma(t))\leq \int_s^t m(r)\mathrm{d}r \quad \forall s,t \in [0,1],\ s\leq t.
\end{equation}
This curves have the property to be differentiable (in the metric sense) a.e.. To be more precise the following theorem (cf. \cite{Ambrosio}) holds

\begin{thm}\label{metrabl}
Let $\gamma \in AC^2([0,1],\X)$. Then for Lebesgue-a.e. $t\in [0,1]$ there exists the limit
\begin{equation}
|\dot{\gamma}|(t):= \lim_{h \to 0} \frac{d(\gamma(t),\gamma(t+h))}{|h|}.
\end{equation}
Furthermore $|\dot{\gamma}|\in L^2$ and we know $d(\gamma(s),\gamma(t))\leq \int_s^t |\dot{\gamma}|(r)\mathrm{d}r$. Moreover $|\dot{\gamma}|(t)\leq m(t)$ for Lebesgue-a.e. $t \in [0,1]$, for all $m$ such that (\ref{ac2}) holds.
\end{thm}

Now we are able to formulate following lemma

\begin{lemma}\label{acenergy}
For all $\gamma \in \Omega$ we define
\begin{equation}
\widetilde{H}(\gamma):=\left\{ \begin{array}{cl}
                            \frac{1}{2}\int_0^1|\dot{\gamma}|^2(r)\mathrm{d}r & \mathrm{,\,if } \gamma \in AC^2([0,1],\X)\\
                            \infty & \mathrm{,\,else.}
                            \end{array}\right.
\end{equation}
Then $H(\gamma)\leq \widetilde{H}(\gamma)$.
\end{lemma}

\begin{proof}
(i): $\gamma \not \in AC^2 \quad \Longrightarrow \quad H(\gamma)\leq \widetilde{H}(\gamma)=\infty.$ \checkmark\\[2ex]
(ii): $\gamma \in AC^2$: Let $\Delta^n=\{0=t_0<t_1<\ldots<t_n=1\}$ be an arbitrary partition, then
\begin{eqnarray*}
\frac{1}{2}\sum_{i=0}^{n-1}\frac{d^2(\gamma(t_{i}),\gamma(t_{i+1}))}{t_{i+1}-t_i} &=& \frac{1}{2}\sum_{i=0}^{n-1} (t_{i+1}-t_i) \left(\frac{d(\gamma(t_{i}),\gamma(t_{i+1}))}{t_{i+1}-t_i}\right)^2\\
&\leq& \frac{1}{2} \sum_{i=0}^{n-1} (t_{i+1}-t_i)^{-1} \left(\int_{t_i}^{t_{i+1}} |\dot{\gamma}|(r) \mathrm{d}r\right)^2\\
&\leq&\frac{1}{2}  \sum_{i=0}^{n-1} \int_{t_i}^{t_{i+1}} |\dot{\gamma}|^2(r) \mathrm{d}r = \frac{1}{2} \int_0^1 |\dot{\gamma}|^2(r) \mathrm{d}r=\widetilde{H}(\gamma).
\end{eqnarray*}
Since this holds true for all partitions we get $H(\gamma)\leq \widetilde{H}(\gamma)$.
\end{proof}

\begin{rem}
With the notation from above the lower bound of the weak LDP of the law of the rescaled Markov process $X^s_\cdot=X_{s\cdot}$ reads as
$$\liminf_{s\searrow 0} s \log \Prob\left(X^s_\cdot \in \Gamma\right) \geq - \inf_{\gamma \in \Gamma^\circ} \widetilde{H}(\gamma).$$
\end{rem}

The next goal is to prove equality in the conclusion of lemma \ref{acenergy}, namely

\begin{thm}
Let $\gamma \in \Omega$ and $H(\gamma)$ and $\widetilde H(\gamma)$ defined as above. Then
$$H(\gamma)=\widetilde{H}(\gamma).$$
\end{thm}

\begin{proof}
It remains to show $\widetilde{H}(\gamma)\leq H(\gamma)$. First of all we observe that if $\gamma \notin AC^2$ then $\sup_{\Delta^n} \sum_{i=0}^{n-1} d(\gamma(t_i),\gamma(t_{i+1}))=\infty$ and hence also $\sup_{\Delta^n} \sum_{i=0}^{n-1} \frac{d^2(\gamma(t_i),\gamma(t_{i+1}))}{t_{i+1}-t_i}=\infty$. Consequently we know $\widetilde{H}(\gamma)=\infty = H(\gamma)$ for all $\gamma \notin AC^2$.

On the other hand if $\gamma \in AC^2$ we see $\sup_{\Delta^n} \sum_{i=0}^{n-1} d(\gamma(t_i),\gamma(t_{i+1})) \leq \sup_{\Delta^n} \int_0^1 m(r) < \infty$ where $m$ is a $L^2$ function (cf. (\ref{ac2})). So in the following considerations it is adequate only to take care about continuous curves $\gamma$ with finite length.

For such a $\gamma$ we define the discrete measure
$$\nu_N:=\sum_{i=0}^{N-1} d\left(\gamma \left( \frac{i}{N}\right),\gamma \left(\frac{i+1}{N} \right)\right).$$
This bounded monotone sequence converges up to subsequences to a measure $\nu$ for $N \to \infty$. Further we know
$$d(\gamma(s),\gamma(t)) \leq \nu_N([s,t]) \quad \mbox{for all } 0\leq s \leq t \leq1 \mbox{ and } s,t \in \left\{0,\frac{1}{N},\ldots,\frac{N-1}{N},1\right\}.$$
Passing to the limit yields
\begin{equation}\label{dleqnu}
d(\gamma(s),\gamma(t)) \leq \nu([s,t]) \quad \mbox{for all } 0\leq s \leq t \leq1.
\end{equation}

Consider
\begin{equation}
\mathcal{E}(\nu|\mu):=\left\{ \begin{array}{cl}
                            \int_0^1|\frac{\mathrm{d}\,\nu}{\mathrm{d}\mu}|^2\mathrm{d}\,\mu & \mathrm{,if\ } \nu \ll \mu\\
                            \infty & \mathrm{,else.}
                            \end{array}\right.
\end{equation}
This is a joint semicontinuous functional.

Let
$$\mu_N:=\sum_{i=0}^{N-1}\frac{1}{N}\delta_{i/N} \rightharpoonup \mu$$
where $\mu$ is the Lebesgue measure on $[0,1]$. Then

$$\mathcal{E}(\nu_N|\mu_N)=\int_0^1 |\frac{\mathrm{d}\,\nu_N}{\mathrm{d}\,\mu_N}|^2\mathrm{d}\mu_N=\sum_{i=0}^{N-1} \frac{d^2\left(\gamma\left(\frac{i}{N}\right),\gamma\left(\frac{i+1}{N}\right)\right)}{1/N}\leq \sup_{\Delta^n} \sum_{i=0}^{n-1}\frac{d^2(\gamma(t_i),\gamma(t_{i+1}))}{t_{i+1}-t_i}=2H(\gamma).$$

So if $H(\gamma) < \infty$ then also $\mathcal{E}(\nu|\mu)< \infty$ and therefore $\nu$ is absolutely continuous with respect to the Lebesgue measure $\mu$. To be more precise $\nu=f\,\mu$ with $||f||_2\leq \sqrt{2H(\gamma)}.$

Together with (\ref{dleqnu}) we see
$$d(\gamma(s),\gamma(t))\leq \nu([s,t])=\int_s^t f(r)\mathrm{d}r.$$
Then theorem \ref{metrabl} yields $|\dot{\gamma}|_r \leq f(r)$ for Lebesgue-a.e. $r \in [0,1]$. Hence we get for $\gamma \in AC^2$

$$\int_0^1|\dot{\gamma}|^2_r \mathrm{d}r \leq \int_0^1 f(r)^2 \mathrm{d}r \leq\int_0^1 \sup_{\Delta^n}\sum_{i=0}^{n-1}\frac{d^2(\gamma(t_i),\gamma(t_{i+1}))}{t_{i+1}-t_i}\,\mathrm{d}r=2H(\gamma).$$
\end{proof}

The argument of the last proof was communicated to us by Professor L. Ambrosio.

Now we are able to state our main theorem

\begin{thm}
Let $\Omega=\mathcal{C}([0,1],\X)$. Let $\Delta^n$ as above a arbitrary partition of the unit interval $[0,1]$ and for all $s >0$ let $X^s_\cdot=\{X_{s\cdot t}\}_{0\leq t \leq 1}$ be the rescaled Markov process. Then we have the following estimates
\begin{itemize}
\item[(i)] For all $\alpha > 0$ and all compact subsets $\Gamma$ of $\{\gamma \in \Omega : \widetilde{H}(\gamma) > \alpha\}$ we have
$$\limsup_{s \to 0} s \log \Prob(X^s_\cdot \in \Gamma) \leq - \alpha.$$
\item[(ii)] For all $\gamma \in \{\gamma : \widetilde{H}(\gamma)<\infty\}$ and all measurable $\Gamma$ with $\gamma \in \Gamma^{\circ}$ we have
$$\limsup_{s \to 0} s \log \Prob(X^s_\cdot \in \Gamma) \geq - \widetilde{H}(\gamma).$$
\end{itemize}
\end{thm}

%
%
%
%
%
%
%
%
%

\bibliographystyle{plain}
\bibliography{lit}

\begin{thebibliography}{1}

\bibitem{Ambrosio}
Luigi Ambrosio, Nicola Gigli, and Giuseppe Savar{\'e}.
\newblock {\em Gradient flows in metric spaces and in the space of probability
  measures}.
\newblock Lectures in Mathematics ETH Z\"urich. Birkh\"auser Verlag, Basel,
  second edition, 2008.

\bibitem{dembo}
Amir Dembo and Ofer Zeitouni.
\newblock {\em Large deviations techniques and applications}, volume~38 of {\em
  Applications of Mathematics (New York)}.
\newblock Springer-Verlag, New York, second edition, 1998.

\bibitem{HR}
M.~Hino and J.~Ram\'{\i}rez.
\newblock Analysis on local {S}mall--time {G}aussian behavior of symmetric
  diffusion semigroups.
\newblock {\em ANN. Probab.}, 75(3):273--297, 1996.

\bibitem{SII}
K.~T. Sturm.
\newblock Analysis on local {D}irichlet spaces. {II}. {U}pper {G}aussian
  estimates for the fundamental solutions of parabolic equations.
\newblock {\em J. Math. Pures Appl. (9)}, 75(3):273--297, 1996.

\bibitem{SIII}
K.~T. Sturm.
\newblock Analysis on local {D}irichlet spaces. {III}. {T}he parabolic
  {H}arnack inequality.
\newblock {\em J. Math. Pures Appl. (9)}, 75(3):273--297, 1996.

\end{thebibliography}

\end{document}